\documentclass[11pt]{article}

\usepackage[T1]{fontenc}
\usepackage[utf8]{inputenc}
\usepackage[english]{babel}
\usepackage{amsmath,amssymb,amsfonts,amsthm,mathtools}
\usepackage[nopatch=true,expansion=false]{microtype}
\usepackage[margin=1.15in]{geometry}
\usepackage[colorlinks=true,linkcolor=blue,citecolor=blue,urlcolor=blue]{hyperref}
\usepackage{xcolor}
\usepackage{graphicx}
\DeclareMathOperator{\tr}{tr}
\DeclareMathOperator{\diver}{div}
\newcommand{\R}{\mathbb R}
\newcommand{\C}{\mathbb C}
\newcommand{\dd}{\,d}
\newcommand{\hOm}{\widehat\Omega}
\newcommand{\hG}{\widehat\Gamma}
\newcommand{\hu}{\widehat u}
\newcommand{\hv}{\widehat v}
\newcommand{\hp}{\widehat p}
\newcommand{\hf}{\widehat f}
\newcommand{\hq}{\widehat q}
\newcommand{\hx}{\widehat x}
\newcommand{\hy}{\widehat y}
\newcommand{\cD}{\mathcal D}
\newcommand{\cF}{\mathcal F}
\newcommand{\cP}{\mathcal P}
\newcommand{\cX}{\mathcal X}
\newcommand{\cY}{\mathcal Y}
\newcommand{\dual}[2]{\left\langle #1,#2\right\rangle}

\theoremstyle{plain}
\newtheorem{theorem}{Theorem}[section]
\newtheorem{lemma}[theorem]{Lemma}
\newtheorem{proposition}[theorem]{Proposition}
\newtheorem{corollary}[theorem]{Corollary}
\theoremstyle{definition}
\newtheorem{definition}[theorem]{Definition}
\theoremstyle{remark}
\newtheorem{remark}[theorem]{Remark}

\title{Analyticity of the data-to-solution map for a stationary Navier--Stokes fluid--structure interaction problem}

\author{Iva Miku\v{s}\thanks{University of Zagreb Faculty of Electrical Engineering and Computing. Email: \href{mailto:Iva.Mikus@fer.unizg.hr}{Iva.Mikus@fer.unizg.hr}.}
\and Boris Muha\thanks{University of Zagreb Faculty of Science, Department of Mathematics. Emails: \href{mailto:borism@math.hr}{borism@math.hr}, \href{mailto:boris.muha@math.pmf.unizg.hr}{boris.muha@math.pmf.unizg.hr}.}}

\date{}

\begin{document}
\maketitle

\begin{abstract}
We consider a stationary fluid--structure interaction problem in which the steady Navier--Stokes equations are coupled, through a free elastic interface, with a clamped Euler--Bernoulli beam equation. Using a complexification of the fixed-domain formulation and the holomorphic implicit function theorem, we prove that, in a neighbourhood of the trivial solution, the mapping from the right-hand side to the weak solution is real analytic. As a byproduct we obtain a small-data existence and local uniqueness result for the coupled system. Our motivation comes from data-driven reduced-order modelling for parametric PDEs, where approximation properties are closely related to the regularity of the solution map. Numerically, a manufactured-solution test exhibits approximately second-order convergence in the reported relative $L^2$ errors, while a proper orthogonal decomposition study for a parametric force family shows rapid decay of the empirical reconstruction error until a numerical floor is reached.
\end{abstract}

\section{Introduction}

The construction of efficient surrogate models for parametric partial differential equations is an important problem in computational science and engineering. Given a parameter $\mu$ in some admissible set, one seeks fast and accurate approximations of the solution map $\mu\mapsto u(\mu)$, with applications ranging from many-query simulation and uncertainty quantification to optimization and inverse problems. Among the various approaches, data-driven reduced-order models build low-dimensional approximations from a finite collection of high-fidelity solutions $\{u(\mu_i)\}_{i=1}^N$. Classical linear methods include proper orthogonal decomposition and reduced-basis methods, while nonlinear alternatives include deep-learning based reduced-order models; see \cite{c01a2bc6-f9f6-36a9-835f-6f8e5dd97c30,quarteroni2015reduced,Franco_2022, Solera-Rico2024} and the references therein.

The approximation properties of such methods depend on more than regularity alone, but regularity of the solution map is a basic ingredient in quantitative approximation theory. For finite-dimensional parameter sets, a quantitatively controlled holomorphic extension can yield exponential-type approximation rates, with the exponent depending on the parameter dimension; for countably parametric problems, rates additionally depend on sparsity or summability assumptions. Corresponding polynomial and neural-network approximation results are discussed in \cite{Cohen2011ANALYTICRA,doi:10.1142/S0219530518500203}. At the weaker level of finite smoothness, regularity also enters neural-network approximation estimates; see, for example, \cite{yarotsky2017errorboundsapproximationsdeep}. Thus analyticity is a useful structural property, but by itself it does not imply a dimension-independent exponential decay of Kolmogorov widths.

Our motivation is to apply this methodology to fluid--structure interaction (FSI) problems. In recent work \cite{BukacManojlovicMuhaVlah2024,MikusMuhaVlah2026}, deep-learning based reduced-order models were developed for parametric PDEs that include geometrically parametrized examples. The numerical experiments there suggest that related methodologies may be useful for FSI problems, where the fluid domain is determined by the unknown structure displacement. However, theoretical foundations for FSI reduced-order models are less developed than for problems posed on a fixed domain. A first step is to analyze the regularity of the data-to-solution map for a representative FSI problem, which is the goal of the present paper.

The question addressed here is not merely dependence of solutions on a prescribed geometry. For the steady Stokes and Navier--Stokes equations on parametrized families of domains, smooth and analytic dependence on the shape parameter has been studied extensively; see, for example, \cite{Sokoowski1992IntroductionTS,doi:10.1137/1.9780898719826,doi:10.1137/16M1099406}. In our setting the domain is itself an unknown of the coupled system, determined self-consistently by fluid--structure equilibrium. The parameter is the volume force acting on the fluid, and $(u,p,\eta)$ denote, respectively, the fluid velocity, the fluid pressure, and the beam displacement. We study the dependence of this coupled configuration on the force; this dependence contains both the geometric nonlinearity of the moving boundary and the nonlinear fluid--structure coupling.

We consider, as a benchmark problem, the steady incompressible Navier--Stokes equations on a planar channel whose top boundary is a linearly elastic beam. The beam is clamped at the endpoints and the fluid--structure coupling is realized through the fluid traction on the deformed interface. The geometry of the fluid domain is therefore an unknown of the problem, determined by the beam displacement. Our main result states that, in a neighbourhood of the trivial solution, the mapping that associates to each volume force the corresponding weak solution $(u,p,\eta)$ is real analytic. The proof is based on the holomorphic implicit function theorem applied on a fixed reference domain, after a pull-back that separates the geometric nonlinearity from the functional-analytic structure of the equations. As a byproduct we obtain a small-data existence and local uniqueness result for the coupled Navier--Stokes--beam system. The analysis must simultaneously handle the free-boundary character of the fluid domain and the dimensional mismatch between the bulk fluid and the lower-dimensional structure.

\subsection*{Related work}

Steady FSI with bulk elastic structures was studied in the foundational paper of Grandmont \cite{grandmont2002}, where existence of a regular solution is proved for small data by a fixed-point argument applied to the coupled system in a Lagrangian formulation. A small-data existence and uniqueness theory for flow past a deformable elastic body in an exterior domain was developed by Galdi and Kyed \cite{galdiKyed2009}. Our analysis fits in the same broad framework of a pull-back to a reference configuration, smallness assumptions, variational estimates, and an implicit-function argument, but the lower-dimensional structure and the function-valued parameter require a careful weak definition of the fluid traction on the elastic interface.

Stationary FSI problems with lower-dimensional elastic structures have been studied in a number of geometric configurations. An early example is the coupling of a two-dimensional inviscid flow with a nonlinear elastic ring in \cite{deCristoforisAntman1991}. Bayada, Chambat, Cid and V\'azquez \cite{bayadaEtAl2004} prove existence for a stationary coupling between an incompressible Stokes fluid and a one-dimensional elastic rod model representing the upper wall of a lubrication device. The fluid equation is Stokes rather than Navier--Stokes and the coupling involves only the pressure load, but the geometry of the problem and the role of the elastic displacement in determining the fluid domain are similar to ours. Related stationary tube-wall interaction with a three-dimensional elastic wall is studied in \cite{surulescu2007}. More recent bridge-deck and beam configurations in stationary Navier--Stokes flow are considered in \cite{gazzolaPatriarca2022,berchioEtAl2024,biancaBocchiGazzola2025}. The unsteady analogue of the problem we consider, namely incompressible Navier--Stokes coupled to a plate or shell, has been intensively studied in the past two decades; see, for example, \cite{chambolleEtAl2005,grandmont2008,MuhaCanic2013}.

A second body of work that is methodologically close to ours treats stationary FSI problems on a fixed reference domain after an ALE-type transformation, typically in the context of optimal control, sensitivity, or shape calculus. Lasiecka, Szulc and \.{Z}ochowski \cite{lasieckaEtAl2018} prove well-posedness for small stationary solutions of the incompressible Navier--Stokes equations coupled to linearized elasticity under mixed boundary conditions. Differentiability properties of stationary FSI solutions with respect to problem data are established in \cite{wickWollner2019,hintermuellerKroner2023}, and shape sensitivity is analyzed in \cite{scheidSokolowski2018,calistiLucardesiScheid2023}. These papers are concerned with bulk or three-dimensional linear elasticity rather than with a lower-dimensional beam model, and their results address differentiability or shape differentiability. The contribution established here is analyticity of the data-to-solution map with respect to an infinite-dimensional right-hand side.

The remainder of the paper is organized as follows. In Section~\ref{sec:formulation} we introduce the FSI problem, fix notation, give the weak formulation on the reference domain, and state the main result. In Section~\ref{sec:prelim} we recall the analytic facts used in the proof. In Section~\ref{sec:operator} we define the complexified nonlinear operator and prove its holomorphy. In Section~\ref{sec:linearization} we compute the linearized operator at the trivial solution and prove its invertibility. In Section~\ref{sec:proof-main} we apply the holomorphic implicit function theorem. Section~\ref{sec:numerics} is reserved for the numerical experiments.

\section{Problem formulation and main result}\label{sec:formulation}

\subsection{Geometry and the FSI system}

Let $L>0$ and let
\[
  \hOm=(0,L)\times(0,1)\subset\R^2
\]
be the reference fluid domain. Its boundary is decomposed as
\[
  \partial\hOm=\hG_{\rm top}\cup\hG_{\rm b},
  \qquad
  \hG_{\rm top}:=(0,L)\times\{1\},
  \qquad
  \hG_{\rm b}:=\partial\hOm\setminus\hG_{\rm top}.
\]
The part $\hG_{\rm top}$ is the elastic interface and $\hG_{\rm b}$ is the rigid part of the boundary.

Given a function $\eta:(0,L)\to\R$ with $\eta(0)=\eta(L)=0$ and $1+\eta>0$ on $(0,L)$, define
\begin{equation}\label{eq:Teta}
  T_\eta(\hx,\hy):=\bigl(\hx,(1+\eta(\hx))\hy\bigr),
\end{equation}
and the deformed fluid domain
\begin{equation}\label{eq:Omegaeta}
  \Omega(\eta):=T_\eta(\hOm)
  =\{(x,y)\in\R^2:0<x<L,\ 0<y<1+\eta(x)\}.
\end{equation}
The deformed top boundary $\Gamma_{\rm top}(\eta)$ is the graph of $1+\eta$. A direct computation gives
\begin{equation}\label{eq:FJ}
  F_\eta:=\nabla T_\eta
  =
  \begin{pmatrix}
    1&0\\
    \hy\eta'&1+\eta
  \end{pmatrix},
  \qquad
  J_\eta:=\det F_\eta=1+\eta,
\end{equation}
and
\begin{equation}\label{eq:Aeta}
  A_\eta:=F_\eta^{-1}
  =
  \begin{pmatrix}
    1&0\\[1mm]
    -\dfrac{\hy\eta'}{1+\eta}&\dfrac{1}{1+\eta}
  \end{pmatrix}.
\end{equation}
Whenever $\eta\in H^2_0(0,L)$ and $1+\eta\geq\kappa>0$, the map $T_\eta$ is a bi-Lipschitz $C^1$-diffeomorphism from $\hOm$ onto $\Omega(\eta)$.

The parameter of the problem is a volume force $\hf\in L^2(\hOm)^2$, prescribed on the reference domain. The corresponding force in the deformed configuration is
\[
  f_\eta:=\hf\circ T_\eta^{-1}:\Omega(\eta)\to\R^2.
\]
The coupled stationary FSI problem reads as follows: find the fluid velocity $u:\Omega(\eta)\to\R^2$, the pressure $p:\Omega(\eta)\to\R$, the beam displacement $\eta:(0,L)\to\R$, and a constant $c\in\R$ such that
\begin{equation}\label{eq:strong}
\begin{cases}
  -\diver \sigma(u,p)+(u\cdot\nabla)u=f_\eta,
     & \text{in }\Omega(\eta),\\
  \diver u=0,
     & \text{in }\Omega(\eta),\\
  u=0,
     & \text{on }\partial\Omega(\eta),\\
  \partial_x^4\eta=(-\sigma(u,p)n+cn)\cdot e_2\sqrt{1+(\eta')^2},
     & \text{on }(0,L),\\
  \eta(0)=\eta'(0)=\eta(L)=\eta'(L)=0,\\
  \displaystyle\int_0^L\eta(x)\dd x=0.
\end{cases}
\end{equation}
Here $e_2=(0,1)^T$, $n$ is the outer unit normal to $\Omega(\eta)$ on $\Gamma_{\rm top}(\eta)$, and
\begin{equation}\label{eq:cauchy-stress}
  \sigma(u,p):=2D(u)-pI,
  \qquad
  D(u):=\frac12\bigl(\nabla u+(\nabla u)^T\bigr),
\end{equation}
is the Cauchy stress tensor. With $n$ oriented outward from the fluid the density of force exerted by the fluid on the beam is $-\sigma(u,p)n$. The scalar $c$ is the Lagrange multiplier associated with the volume constraint. On the graph $y=1+\eta(x)$ one has
\[
  n\sqrt{1+(\eta')^2}=(-\eta',1)^T,
\]
and therefore the multiplier contributes $c\int_0^L\varphi\dd x$ to the weak beam equation.

\subsection{Weak formulation on the reference domain}

We pull the system back to the fixed reference domain. Given $u$ and $p$ on $\Omega(\eta)$, set
\[
  \hu:=u\circ T_\eta,
  \qquad
  \hp:=p\circ T_\eta.
\]
We use the convention $[\nabla u]_{ij}=\partial_j u_i$. The chain rule gives
\[
  (\nabla u)\circ T_\eta=\nabla\hu\,A_\eta.
\]
Accordingly, for a vector field $\hv$ on $\hOm$ we write
\begin{equation}\label{eq:transformed-operators}
  \nabla_\eta\hv:=\nabla\hv A_\eta,
  \qquad
  D_\eta(\hv):=\frac12\bigl(\nabla_\eta\hv+(\nabla_\eta\hv)^T\bigr),
  \qquad
  \diver_\eta\hv:=\tr(\nabla\hv A_\eta).
\end{equation}
Then
\[
  D(u)\circ T_\eta=D_\eta(\hu),
  \qquad
  (\diver u)\circ T_\eta=\diver_\eta\hu.
\]
We set
\[
  V:=H^2_0(0,L),
  \qquad
  Q:=L^2_0(\hOm):=\left\{q\in L^2(\hOm):\int_{\hOm}q\dd\hx=0\right\},
\]
and
\begin{equation}\label{eq:state-space}
  \cX:=H^1_0(\hOm)^2\times Q\times V\times\R.
\end{equation}
The pressure is normalized by the condition $\hp\in Q$ on the reference domain. This is a normalization convention; it differs from the direct pullback of the physical zero-mean condition only by an additive constant in the pressure. Such a change is compensated by the opposite change in $c$, since the pressure contribution to the interface load depends on the combination $c+\hp$.

To define the normal Cauchy traction with respect to the outer fluid normal, we first fix the trace space and a lifting. Let
\[
  H^{1/2}_{00}(0,L)
  :=
  \{\gamma z_{|_{\hG_{\rm top}}}:
  z\in H^1(\hOm),\ \gamma z_{|\hG_{\rm b}}=0\},
\]
endowed with the quotient trace norm, where $\gamma$ is trace operator. The subscript $00$ is the standard Lions--Magenes notation at the critical exponent $1/2$: it records the endpoint compatibility inherited from extensions that vanish on the adjacent sides. In the present rectangular geometry, a convenient bounded right inverse of the top trace is given explicitly as follows. If
\[
  \varphi(x)=\sum_{k=1}^\infty \varphi_k
  \sin\frac{k\pi x}{L}
\]
is the sine expansion of $\varphi\in H^{1/2}_{00}(0,L)$, set
\begin{equation}\label{eq:explicit-lifting}
  \mathcal E\varphi(x,y)
  :=
  \sum_{k=1}^\infty
  \varphi_k
  \frac{\sinh\left(\frac{k\pi}{L}y\right)}
       {\sinh\left(\frac{k\pi}{L}\right)}
  \sin\frac{k\pi x}{L},
  \qquad
  \cD_\varphi:=\mathcal E\varphi\, e_2 .
\end{equation}
Thus $\mathcal E\varphi$ is the harmonic extension of the top boundary datum $\varphi$ with zero trace on the bottom and vertical sides. In particular,
\[
  \cD_\varphi=\varphi e_2\quad\text{on }\hG_{\rm top},
  \qquad
  \cD_\varphi=0\quad\text{on }\hG_{\rm b},
\]
and the standard trace theorem, or directly the sine-series formula above, gives
\begin{equation}\label{eq:lifting-estimate}
  \|\cD_\varphi\|_{H^1(\hOm)}
  \leq C\|\varphi\|_{H^{1/2}_{00}(0,L)}.
\end{equation}
In particular, for $\varphi\in V=H^2_0(0,L)$,
\begin{equation}\label{eq:lifting-estimate-H2}
  \|\cD_\varphi\|_{H^1(\hOm)}
  \leq C\|\varphi\|_{H^2(0,L)}.
\end{equation}
Since the Cauchy stress is only in $L^2$ for weak solutions, its normal trace is interpreted through this lifting. For $\eta\in V$ with $1+\eta\geq\kappa>0$ define
\begin{equation}\label{eq:traction-functional}
\begin{aligned}
  \mathfrak T(\eta,\hf,\hu,\hp)(\varphi)
  &:={}
  \int_{\hOm}2D_\eta(\hu):D_\eta(\cD_\varphi)J_\eta\dd\hx
  \\
  &\quad+
  \int_{\hOm}\bigl((A_\eta\hu)\cdot\nabla\bigr)\hu\cdot\cD_\varphi J_\eta\dd\hx
  \\
  &\quad-
  \int_{\hOm}\hp\,\diver_\eta\cD_\varphi J_\eta\dd\hx
  -
  \int_{\hOm}\hf\cdot\cD_\varphi J_\eta\dd\hx.
\end{aligned}
\end{equation}
This functional is well-defined for $\hf\in L^2(\hOm)^2$, $\hu\in H^1_0(\hOm)^2$, $\hp\in Q$, and $\varphi\in V$, because $H^1(\hOm)\hookrightarrow L^4(\hOm)$ in two dimensions. The coupling between fluid and beam is realized via the normal trace of the Cauchy stress on the interface. For a weak solution the stress is only $L^2$, so it has no \emph{pointwise} normal trace. It does, however, possess a weak normal trace
in $H^{-1/2}$: for solutions the stress has square-integrable divergence, and
integrating the momentum equation by parts against a lifting of the boundary
datum gives a well-defined dual pairing. The functional $\mathfrak T$ in~\eqref{eq:traction-functional} realizes this weak trace, and crucially for the argument, its formula is
meaningful for every $\hat u\in H^1_0(\hat\Omega)^2$ and $\hat p\in L^2_0(\hat\Omega)$,
not only for solutions, which is what allows it to enter the map to which the
implicit function theorem is applied.

\begin{definition}\label{def:weak-solution}
Let $\hf\in L^2(\hOm)^2$. A quadruple $(\hu,\hp,\eta,c)\in\cX$ with $1+\eta\geq\kappa>0$ is a weak solution of \eqref{eq:strong} if, for every
\[
  (\hv,\hq,\varphi)\in H^1_0(\hOm)^2\times Q\times V,
\]
one has
\begin{equation}\label{eq:weak-fluid}
\begin{aligned}
  &\int_{\hOm}2D_\eta(\hu):D_\eta(\hv)J_\eta\dd\hx
  +\int_{\hOm}\bigl((A_\eta\hu)\cdot\nabla\bigr)\hu\cdot\hv J_\eta\dd\hx
  -\int_{\hOm}\hp\,\diver_\eta\hv J_\eta\dd\hx
  \\
  &\qquad=
  \int_{\hOm}\hf\cdot\hv J_\eta\dd\hx,
\end{aligned}
\end{equation}
\begin{equation}\label{eq:weak-div}
  \int_{\hOm}\diver_\eta\hu\,\hq J_\eta\dd\hx=0,
\end{equation}
\begin{equation}\label{eq:weak-beam}
  \int_0^L\eta''\varphi''\dd x
  =
  -\mathfrak T(\eta,\hf,\hu,\hp)(\varphi)
  +c\int_0^L\varphi\dd x,
\end{equation}
and
\begin{equation}\label{eq:weak-volume}
  \int_0^L\eta\dd x=0.
\end{equation}
\end{definition}
\begin{remark}
  It is enough to test \eqref{eq:weak-div} only with $\hq\in Q$, rather than with arbitrary functions in $L^2(\hOm)$. Indeed, by the divergence theorem and the fact that $\hu\in H^1_0(\hOm)^2$, we have
  \[
    \int_{\hOm} J_\eta\,\diver_\eta \hu\,\dd\hx = 0.
  \]
  Hence $J_\eta\,\diver_\eta\hu=0$ in $L^2_0(\hOm)$ implies $\diver_\eta\hu=0$ in $L^2(\hOm)$.
\end{remark}

\begin{remark}\label{rmk:stress-trace}
The formula \eqref{eq:traction-functional} is the reference-domain analogue of the normal trace of the Cauchy stress on the moving interface. If $(\hu,\hp)$ satisfies the fluid equations \eqref{eq:weak-fluid}--\eqref{eq:weak-div}, then the value of the stress functional is independent of the particular extension used to realize the boundary datum $\varphi e_2$ on $\hG_{\rm top}$. Indeed, the difference of two such extensions belongs to $H^1_0(\hOm)^2$, and testing \eqref{eq:weak-fluid} with this difference gives zero. Thus \eqref{eq:traction-functional} is a rigorous weak definition of
\[
  \int_0^L (\sigma(u,p)n)\cdot e_2\sqrt{1+(\eta')^2}\,\varphi\dd x.
\]
Accordingly, the force exerted by the fluid on the beam is represented by $-\mathfrak T(\eta,\hf,\hu,\hp)$.
This is the standard way of defining normal stresses for low-regularity fluid solutions; see, for instance, \cite[Chapter III]{galdi-book}.
\end{remark}

\subsection{Weak stress trace and regularity}\label{subsec:regularity}

We record explicitly the regularity of the normal stress functional. Although the beam equation in Definition~\ref{def:weak-solution} is tested only against functions in $V=H^2_0(0,L)$, the fluid stress acting on the interface is in fact an element of $H^{-1/2}$.

\begin{proposition}\label{prop:traction-Hminus-half}
Let $\eta\in V$ satisfy $1+\eta\geq\kappa>0$ and $\|\eta\|_{H^2(0,L)}\leq D_0$, let $\hf\in L^2(\hOm)^2$, $\hu\in H^1_0(\hOm)^2$, and $\hp\in L^2_0(\hOm)$. Then the map
\[
  \varphi\mapsto \mathfrak T(\eta,\hf,\hu,\hp)(\varphi)
\]
initially defined for $\varphi\in V$ extends uniquely to a bounded linear functional on $H^{1/2}_{00}(0,L)$. More precisely,
\begin{equation}\label{eq:traction-Hminus-half-estimate}
  |\mathfrak T(\eta,\hf,\hu,\hp)(\varphi)|
  \leq
  C_{D_0,\kappa}
  \bigl(
    \|\hu\|_{H^1(\hOm)}
    +\|\hu\|_{H^1(\hOm)}^2
    +\|\hp\|_{L^2(\hOm)}
    +\|\hf\|_{L^2(\hOm)}
  \bigr)
  \|\varphi\|_{H^{1/2}_{00}(0,L)} .
\end{equation}
Consequently, whenever the fluid equations hold, the normal Cauchy stress on the interface belongs to
\[
  (H^{1/2}_{00}(0,L))',
\]
which we henceforth denote by $H^{-1/2}(0,L):=(H^{1/2}_{00}(0,L))'$.
\end{proposition}

\begin{proof}
We estimate each term in \eqref{eq:traction-functional}. By the assumptions on $\eta$, the coefficients $A_\eta$ and $J_\eta$ are bounded in $L^\infty$, with bounds depending only on $D_0$ and $\kappa$. First,
\[
  \left|
  \int_{\hOm}2D_\eta(\hu):D_\eta(\cD_\varphi)J_\eta\dd\hx
  \right|
  \leq C_{D_0,\kappa}\|\hu\|_{H^1(\hOm)}\|\cD_\varphi\|_{H^1(\hOm)}.
\]
Similarly,
\[
  \left|
  \int_{\hOm}\hp\,\diver_\eta\cD_\varphi J_\eta\dd\hx
  \right|
  \leq C_{D_0,\kappa}\|\hp\|_{L^2(\hOm)}\|\cD_\varphi\|_{H^1(\hOm)},
\]
and
\[
  \left|
  \int_{\hOm}\hf\cdot\cD_\varphi J_\eta\dd\hx
  \right|
  \leq C_{D_0,\kappa}\|\hf\|_{L^2(\hOm)}\|\cD_\varphi\|_{H^1(\hOm)}.
\]
For the convective term we use the two-dimensional embedding $H^1(\hOm)\hookrightarrow L^4(\hOm)$:
\[
\begin{aligned}
  \left|
  \int_{\hOm}\bigl((A_\eta\hu)\cdot\nabla\bigr)\hu\cdot\cD_\varphi J_\eta\dd\hx
  \right|
  &\leq C_{D_0,\kappa}\|\hu\|_{L^4(\hOm)}\|\nabla\hu\|_{L^2(\hOm)}\|\cD_\varphi\|_{L^4(\hOm)}  \\
  &\leq C_{D_0,\kappa}\|\hu\|_{H^1(\hOm)}^2\|\cD_\varphi\|_{H^1(\hOm)}.
\end{aligned}
\]
Combining these estimates with \eqref{eq:lifting-estimate} proves \eqref{eq:traction-Hminus-half-estimate}. Since $V$ is dense in $H^{1/2}_{00}(0,L)$, the extension is unique.
\end{proof}

The preceding proposition gives additional regularity for the structure. Indeed, for a weak solution the beam equation can be written distributionally as
\[
  \eta'''' = -\mathfrak T(\eta,\hf,\hu,\hp)+c
  \quad\text{in }(0,L),
\]
with clamped boundary conditions. Since $c$ defines an $L^2(0,L)$, hence $H^{-1/2}(0,L)$, functional, the right-hand side belongs to $H^{-1/2}(0,L)$. Standard one-dimensional elliptic regularity for the clamped biharmonic operator therefore yields
\begin{equation}\label{eq:beam-H72}
  \eta\in H^{7/2}(0,L),
  \qquad
  \|\eta\|_{H^{7/2}(0,L)}
  \leq
  C\left(
    \|\mathfrak T(\eta,\hf,\hu,\hp)\|_{H^{-1/2}(0,L)}+|c|
  \right).
\end{equation}
Thus the structural displacement is substantially smoother than what is needed to set up the weak formulation.

We also indicate what can be inferred for the fluid variables. The proof of Theorem~\ref{thm:main} is carried out entirely in the natural weak space $H^1_0(\hOm)^2\times L^2_0(\hOm)$. The deformed domain has corner points where the elastic graph meets the fixed vertical sides. We do not assume or verify the corner-regularity and compatibility hypotheses that would be needed for a global $H^2\times H^1$ estimate. Consequently, the solution constructed below is asserted only to be a weak fluid solution; no global strong regularity for $(u,p)$ is claimed in the present geometry. Standard interior and boundary regularity away from the corner points nevertheless gives higher local regularity there.

In variants without corner obstructions, for example with periodicity in the horizontal direction or on a smooth reference boundary, the usual Stokes and stationary Navier--Stokes elliptic regularity theory can be invoked; see, for instance, \cite[Chapters IV and IX]{galdi-book}. Under the corresponding hypotheses, $\hf\in L^2$ and sufficiently small data yield
\[
  u\in H^2(\Omega(\eta))^2,
  \qquad
  p\in H^1(\Omega(\eta)),
\]
and smoother data yield higher Sobolev regularity. These statements concern such modified geometries and are not used in the proof of the present theorem.
\subsection{Main result}

For $D>0$, define
\begin{equation}\label{eq:VD}
  V_D:=\{\eta\in V:\|\eta\|_{H^2(0,L)}<D\}.
\end{equation}
We choose $D$ small enough so that
\begin{equation}\label{eq:D-small}
  \|\eta\|_{L^\infty(0,L)}<\frac12
  \qquad\text{for every }\eta\in V_D.
\end{equation}
This is possible by the embedding $H^2(0,L)\hookrightarrow C^1([0,L])$. In particular, $1+\eta\geq1/2$ for all $\eta\in V_D$.

\begin{theorem}\label{thm:main}
There exist $\rho>0$ and an open neighbourhood $\mathcal U$ of $(0,0,0,0)$ in $\cX$ such that, for every $\hf\in L^2(\hOm)^2$ with
\[
  \|\hf\|_{L^2(\hOm)}<\rho,
\]
the weak problem \eqref{eq:weak-fluid}--\eqref{eq:weak-volume} has a unique solution
\[
  (\hu(\hf),\hp(\hf),\eta(\hf),c(\hf))\in\mathcal U.
\]
Moreover, the data-to-solution map
\[
  \hf\mapsto (\hu(\hf),\hp(\hf),\eta(\hf),c(\hf))
\]
is real analytic from a neighbourhood of $0$ in $L^2(\hOm)^2$ into $\cX$. For each such force, the corresponding beam displacement satisfies the additional regularity
\[
  \eta(\hf)\in H^{7/2}(0,L).
\]
\end{theorem}

\begin{remark}
The case $\hf=0$ admits the trivial solution $(\hu,\hp,\eta,c)=(0,0,0,0)$, and Theorem~\ref{thm:main} states that this solution persists, with analytic dependence, under sufficiently small perturbations of the right-hand side. In particular, the theorem gives a small-data local existence and uniqueness result for the coupled problem.
\end{remark}
The proof of Theorem~\ref{thm:main} proceeds as follows. A change of variables
to the fixed reference domain $\hOm$ removes the free boundary and confines the
geometric nonlinearity to explicit coefficients depending on $\eta$. We then
complexify the resulting weak formulation and encode it as a single nonlinear
map $F$ whose zeros are the weak solutions; real analyticity of the
data-to-solution map follows once $F$ is shown to be holomorphic. Holomorphy
reduces to that of the coefficient maps $\eta\mapsto J_\eta,\,A_\eta$ and of the
weak traction functional, which remain holomorphic under complexification even
though complex displacements carry no geometric meaning. Finally, at the trivial
state the linearization $D_X F(0)$ splits into a Stokes saddle-point problem and
a clamped biharmonic problem, coupled through the linearized interface traction;
its invertibility, together with holomorphy, places the problem within the assumptions
of the holomorphic implicit function theorem, which yields both the existence of local
solution and its analytic dependence on the force.
\section{Analytic preliminaries}\label{sec:prelim}

We use standard notions of holomorphy between complex Banach spaces. A map between open subsets of complex Banach spaces is holomorphic if it is locally bounded and complex G\^ateaux differentiable, equivalently if it is Fr\'echet holomorphic. A map between real Banach spaces is real analytic if it is the restriction of a holomorphic map between their complexifications.

We use the holomorphic corollary of the Banach-space implicit
function theorem; see \cite[Section~10.2, Theorem~10.2.1,
p.~270]{dieudonne1969}. The same form is used for holomorphic
parameter-to-solution maps in
\cite[Section~2.3, proof of Theorem~2.5]{cohenDeVore2015}.

\begin{theorem}[Holomorphic implicit function theorem]\label{thm:HIFT}
Let $X,Y,Z$ be complex Banach spaces, let $U\subset X\times Y$ be open, and let $F:U\to Z$ be holomorphic. Suppose that $F(x_0,y_0)=0$ and that
\[
  D_yF(x_0,y_0):Y\to Z
\]
is a bounded complex-linear isomorphism. Then there exist neighbourhoods $U_X$ of $x_0$ and $U_Y$ of $y_0$ and a unique holomorphic map $\Phi:U_X\to U_Y$ such that
\[
  F(x,\Phi(x))=0
\]
for all $x\in U_X$. Moreover, these are the only zeros of $F$ in $U_X\times U_Y$.
\end{theorem}

\begin{lemma}\label{lem:complexification-isomorphism}
Let $X,Y$ be real Banach spaces and let $A:X\to Y$ be a bounded linear isomorphism. Then the complexification
\[
  A_\C:X_\C\to Y_\C,
  \qquad
  A_\C(x_1+i x_2):=Ax_1+iAx_2,
\]
is a bounded complex-linear isomorphism.
\end{lemma}

\begin{proof}
For $y_1+i y_2\in Y_\C$, let $x_j=A^{-1}y_j$. Then $A_\C(x_1+i x_2)=y_1+i y_2$, which proves surjectivity. Injectivity is identical. Boundedness of $A_\C^{-1}$ follows from boundedness of $A^{-1}$.
\end{proof}

Finally, we record the coefficient regularity used throughout. Since $H^2(0,L)\hookrightarrow W^{1,\infty}(0,L)$ and $H^2(0,L)$ is a Banach algebra, after decreasing $D$ if necessary the map
\[
  \eta\mapsto (1+\eta)^{-1}
\]
is holomorphic from the complex ball
\[
  V_{D,\C}:=\{\eta\in H^2_0(0,L;\C):\|\eta\|_{H^2(0,L)}<D\}
\]
into $W^{1,\infty}(0,L;\C)$. Hence the coefficient maps
\begin{equation}\label{eq:coeff-holomorphic}
  \eta\mapsto J_\eta,
  \qquad
  \eta\mapsto A_\eta,
  \qquad
  \eta\mapsto D_\eta,
\end{equation}
understood as multiplier-valued maps, are holomorphic on $V_{D,\C}$. The same holds for all products of these coefficients appearing below.

\section{The complexified nonlinear operator}\label{sec:operator}

Set
\[
  \cP_\C:=L^2(\hOm;\C)^2,
\]
\[
  \cX_\C:=H^1_0(\hOm;\C)^2\times L^2_0(\hOm;\C)\times H^2_0(0,L;\C)\times\C,
\]
and
\[
  \cY_\C:=H^{-1}(\hOm;\C)^2\times (L^2_0(\hOm;\C))'\times H^{-2}(0,L;\C)\times\C.
\]

All dualities are complex-bilinear; no complex conjugates are used. For complex-valued $\eta\in V_{D,\C}$ we keep the algebraic definitions \eqref{eq:FJ}--\eqref{eq:Aeta}. They need not have a geometric interpretation, but they define holomorphic coefficients.

The explicit lifting is extended complex-linearly by the same harmonic formula as in \eqref{eq:explicit-lifting}: if
\[
  \varphi(x)=\sum_{k=1}^\infty \varphi_k
  \sin\frac{k\pi x}{L},
\]
then
\begin{equation}\label{eq:complex-lifting}
  \cD_\varphi(x,y)
  =
  \left(
  \sum_{k=1}^\infty
  \varphi_k
  \frac{\sinh\left(\frac{k\pi}{L}y\right)}
       {\sinh\left(\frac{k\pi}{L}\right)}
  \sin\frac{k\pi x}{L}
  \right)e_2,
  \qquad
  \varphi\in H^2_0(0,L;\C).
\end{equation}
Define
\[
  \cF_\C:\cP_\C\times
  \bigl(H^1_0(\hOm;\C)^2\times L^2_0(\hOm;\C)\times V_{D,\C}\times\C\bigr)
  \to\cY_\C
\]
by
\[
  \cF_\C=(\cF_{\C,1},\cF_{\C,2},\cF_{\C,3},\cF_{\C,4}).
\]
For test functions $\hv\in H^1_0(\hOm;\C)^2$, $\hq\in L^2_0(\hOm;\C)$, and $\varphi\in H^2_0(0,L;\C)$, set
\begin{equation}\label{eq:FC1}
\begin{aligned}
  \dual{\cF_{\C,1}(\hf,\hu,\hp,\eta,c)}{\hv}
  &:={}
  \int_{\hOm}2D_\eta(\hu):D_\eta(\hv)J_\eta\dd\hx
  +\int_{\hOm}\bigl((A_\eta\hu)\cdot\nabla\bigr)\hu\cdot\hv J_\eta\dd\hx
  \\
  &\quad-
  \int_{\hOm}\hp\,\diver_\eta\hv J_\eta\dd\hx
  -
  \int_{\hOm}\hf\cdot\hv J_\eta\dd\hx,
\end{aligned}
\end{equation}
\begin{equation}\label{eq:FC2}
  \dual{\cF_{\C,2}(\hf,\hu,\hp,\eta,c)}{\hq}
  :=
  \int_{\hOm}\diver_\eta\hu\,\hq J_\eta\dd\hx,
\end{equation}
\begin{equation}\label{eq:FC3}
\begin{aligned}
  \dual{\cF_{\C,3}(\hf,\hu,\hp,\eta,c)}{\varphi}
  &:={}
  \int_0^L\eta''\varphi''\dd x
  +\mathfrak T_\C(\eta,\hf,\hu,\hp)(\varphi)
  -c\int_0^L\varphi\dd x,
\end{aligned}
\end{equation}
and
\begin{equation}\label{eq:FC4}
  \cF_{\C,4}(\hf,\hu,\hp,\eta,c):=\int_0^L\eta\dd x.
\end{equation}
Here
\begin{equation}\label{eq:complex-traction}
\begin{aligned}
  \mathfrak T_\C(\eta,\hf,\hu,\hp)(\varphi)
  &:={}
  \int_{\hOm}2D_\eta(\hu):D_\eta(\cD_\varphi)J_\eta\dd\hx
  \\
  &\quad+
  \int_{\hOm}\bigl((A_\eta\hu)\cdot\nabla\bigr)\hu\cdot\cD_\varphi J_\eta\dd\hx
  \\
  &\quad-
  \int_{\hOm}\hp\,\diver_\eta\cD_\varphi J_\eta\dd\hx
  -
  \int_{\hOm}\hf\cdot\cD_\varphi J_\eta\dd\hx.
\end{aligned}
\end{equation}

\begin{proposition}\label{prop:F-holomorphic}
The map $\cF_\C$ is holomorphic in a neighbourhood of the origin in $\cP_\C\times\cX_\C$.
\end{proposition}

\begin{proof}
The coefficient maps in \eqref{eq:coeff-holomorphic} are holomorphic on $V_{D,\C}$. The remaining operations in \eqref{eq:FC1}--\eqref{eq:complex-traction} are continuous linear, bilinear, or trilinear maps between the displayed Banach spaces. For instance, the diffusion term is bounded because $A_\eta$ and $J_\eta$ are bounded multipliers, and
\[
  \left|\int_{\hOm}2D_\eta(\hu):D_\eta(\hv)J_\eta\dd\hx\right|
  \leq C(\eta)\|\hu\|_{H^1}\|\hv\|_{H^1}.
\]
For the convective term we use the two-dimensional embedding $H^1(\hOm)\hookrightarrow L^4(\hOm)$:
\[
  \left|\int_{\hOm}\bigl((A_\eta\hu)\cdot\nabla\bigr)\hu\cdot\hv J_\eta\dd\hx\right|
  \leq C(\eta)\|\hu\|_{H^1}^2\|\hv\|_{H^1}.
\]
The same estimate applies to the convective part of the traction functional, since $\cD_\varphi\in H^1(\hOm)^2$ and \eqref{eq:lifting-estimate} holds. Thus each component of $\cF_\C$ is obtained from holomorphic coefficient maps by continuous multilinear operations, and is therefore holomorphic.
\end{proof}

For real-valued arguments satisfying the fluid equations, the zero set of $\cF_\C$ restricted to the real spaces agrees with Definition~\ref{def:weak-solution}, by Remark~\ref{rmk:stress-trace}.

\section{The linearized operator}\label{sec:linearization}

At the trivial point $(\hf,\hu,\hp,\eta,c)=(0,0,0,0,0)$ one has $J_0=1$ and $A_0=I$. We compute the derivative of $\cF_\C$ with respect to the unknowns $X=(\hu,\hp,\eta,c)$. Let $(w,r,\psi,d)\in\cX_\C$.

We first note explicitly that the convective term does not contribute to the linearization at the origin. Indeed,
\[
  \mathcal N(\eta,\hu;\hv):=\int_{\hOm}\bigl((A_\eta\hu)\cdot\nabla\bigr)\hu\cdot\hv J_\eta\dd\hx
\]
is quadratic in $\hu$. Therefore every first variation of $\mathcal N$ at $(\eta,\hu)=(0,0)$ contains at least one factor equal to the base velocity $\hu=0$. The same argument applies to the convective term in the traction functional. Similarly, all derivatives of the coefficients with respect to $\eta$ in the diffusion, pressure, and forcing terms vanish at the origin when differentiating with respect to $X$, because they are multiplied by the base values $\hu=0$, $\hp=0$, and $\hf=0$.

Consequently,
\begin{equation}\label{eq:lin1}
\begin{aligned}
  \dual{D_X\cF_{\C,1}(0)(w,r,\psi,d)}{\hv}
  =
  \int_{\hOm}2D(w):D(\hv)\dd\hx
  -
  \int_{\hOm}r\,\diver\hv\dd\hx,
\end{aligned}
\end{equation}
\begin{equation}\label{eq:lin2}
  \dual{D_X\cF_{\C,2}(0)(w,r,\psi,d)}{\hq}
  =
  \int_{\hOm}(\diver w)\hq\dd\hx,
\end{equation}
\begin{equation}\label{eq:lin3}
\begin{aligned}
  \dual{D_X\cF_{\C,3}(0)(w,r,\psi,d)}{\varphi}
  =
  \int_0^L\psi''\varphi''\dd x
  +\mathfrak T_0(w,r)(\varphi)
  -d\int_0^L\varphi\dd x,
\end{aligned}
\end{equation}
and
\begin{equation}\label{eq:lin4}
  D_X\cF_{\C,4}(0)(w,r,\psi,d)=\int_0^L\psi\dd x.
\end{equation}
Here
\begin{equation}\label{eq:T0}
  \mathfrak T_0(w,r)(\varphi)
  :=
  \int_{\hOm}2D(w):D(\cD_\varphi)\dd\hx
  -
  \int_{\hOm}r\,\diver\cD_\varphi\dd\hx.
\end{equation}
The term $\mathfrak T_0(w,r)$ is the linearized normal Cauchy traction with respect to the outer fluid normal. It is essential in the coupled linearization.

\begin{proposition}\label{prop:real-invertibility}
The real operator
\[
  D_X\cF(0):\cX\to H^{-1}(\hOm)^2\times Q'\times H^{-2}(0,L)\times\R
\]
is a bounded linear isomorphism.
\end{proposition}

\begin{proof}
Let
\[
  (G,H,\Phi,\gamma)
  \in H^{-1}(\hOm)^2\times Q'\times H^{-2}(0,L)\times\R.
\]
We solve
\[
  D_X\cF(0)(w,r,\psi,d)=(G,H,\Phi,\gamma).
\]
The first two equations are the Stokes saddle-point problem
\begin{equation}\label{eq:lin-stokes}
\begin{aligned}
  \int_{\hOm}2D(w):D(v)\dd\hx
  -\int_{\hOm}r\,\diver v\dd\hx
  &=\dual{G}{v},
  && v\in H^1_0(\hOm)^2,\\
  \int_{\hOm}(\diver w)q\dd\hx
  &=\dual{H}{q},
  && q\in Q.
\end{aligned}
\end{equation}
By Korn's inequality and the standard Stokes inf-sup theorem on bounded Lipschitz domains, \eqref{eq:lin-stokes} has a unique solution
\[
  (w,r)\in H^1_0(\hOm)^2\times Q
\]
and
\begin{equation}\label{eq:stokes-estimate}
  \|w\|_{H^1(\hOm)}+
  \|r\|_{L^2(\hOm)}
  \leq
  C\bigl(\|G\|_{H^{-1}(\hOm)}+\|H\|_{Q'}\bigr).
\end{equation}

With $(w,r)$ fixed, the beam part becomes
\begin{equation}\label{eq:beam-linear}
  \int_0^L\psi''\varphi''\dd x
  -d\int_0^L\varphi\dd x
  =
  \dual{\Phi}{\varphi}-\mathfrak T_0(w,r)(\varphi),
  \qquad \varphi\in V,
\end{equation}
and
\begin{equation}\label{eq:beam-constraint-linear}
  \int_0^L\psi\dd x=\gamma.
\end{equation}
The functional
\[
  \varphi\mapsto\dual{\Phi}{\varphi}-\mathfrak T_0(w,r)(\varphi)
\]
belongs to $V'=H^{-2}(0,L)$. Hence, by the Lax--Milgram theorem, there is a unique $\psi_1\in V$ such that
\[
  \int_0^L\psi_1''\varphi''\dd x
  =
  \dual{\Phi}{\varphi}-\mathfrak T_0(w,r)(\varphi),
  \qquad \varphi\in V.
\]
Let $\psi_0\in V$ be the unique solution of
\[
  \int_0^L\psi_0''\varphi''\dd x
  =
  \int_0^L\varphi\dd x,
  \qquad \varphi\in V.
\]
Equivalently, $\psi_0''''=1$ with clamped boundary conditions, so
\[
  \psi_0(x)=\frac{x^2(L-x)^2}{24},
\]
and therefore
\[
  \int_0^L\psi_0\dd x>0.
\]
Define
\[
  d:=\frac{\gamma-\int_0^L\psi_1\dd x}{\int_0^L\psi_0\dd x},
  \qquad
  \psi:=\psi_1+d\psi_0.
\]
Then $(\psi,d)$ is the unique solution of \eqref{eq:beam-linear}--\eqref{eq:beam-constraint-linear}. The construction and the estimates above give boundedness of the inverse operator. Thus $D_X\cF(0)$ is a bounded linear isomorphism.
\end{proof}

\begin{corollary}\label{cor:complex-invertibility}
The complex-linear operator
\[
  D_X\cF_\C(0):\cX_\C\to\cY_\C
\]
is a bounded linear isomorphism.
\end{corollary}

\begin{proof}
This follows from Proposition~\ref{prop:real-invertibility} and Lemma~\ref{lem:complexification-isomorphism}.
\end{proof}

\section{Proof of the main theorem}\label{sec:proof-main}

We now prove Theorem~\ref{thm:main}. By Proposition~\ref{prop:F-holomorphic}, the map $\cF_\C$ is holomorphic near the origin. Moreover,
\[
  \cF_\C(0,0,0,0,0)=0,
\]
and, by Corollary~\ref{cor:complex-invertibility}, the derivative with respect to the unknowns at the origin is an isomorphism. The holomorphic implicit function theorem, Theorem~\ref{thm:HIFT}, gives neighbourhoods
\[
  \mathcal V_\C\subset\cP_\C,
  \qquad
  \mathcal U_\C\subset\cX_\C,
\]
of the origin and a unique holomorphic map
\[
  \mathcal S_\C:\mathcal V_\C\to\mathcal U_\C
\]
such that
\[
  \cF_\C(\hf,\mathcal S_\C(\hf))=0
  \qquad\text{for all }\hf\in\mathcal V_\C.
\]
Write
\[
  \mathcal S_\C(\hf)=(\hu(\hf),\hp(\hf),\eta(\hf),c(\hf)).
\]
After shrinking the implicit-function neighbourhoods, if necessary, we may take them invariant under complex conjugation. For real-valued $\hf$, the conjugate $\overline{\mathcal S_\C(\hf)}$ is then another zero of the same real-coefficient equation in the uniqueness neighbourhood. Local uniqueness therefore gives $\overline{\mathcal S_\C(\hf)}=\mathcal S_\C(\hf)$, so restriction to real-valued data defines a real analytic map into the real state space $\cX$.

For real-valued arguments, zeros of $\cF_\C$ are precisely weak solutions in the sense of Definition~\ref{def:weak-solution}. The identification of the traction term is justified by Remark~\ref{rmk:stress-trace}. Therefore, after possibly shrinking the neighbourhoods and choosing $\rho>0$ such that
\[
  \{\hf\in L^2(\hOm)^2:\|\hf\|_{L^2}<\rho\}\subset \mathcal V_\C\cap L^2(\hOm)^2,
\]
we obtain the asserted local weak solution and its real analytic dependence on $\hf$.

The uniqueness assertion is local: if $\hf$ belongs to this ball and if a weak solution belongs to the neighbourhood $\mathcal U:=\mathcal U_\C\cap\cX$, then it coincides with $\mathcal S_\C(\hf)$. This is the uniqueness statement supplied by the implicit function theorem. Theorem~\ref{thm:main} follows.

\begin{corollary}\label{cor:zero-unique}
There exists a neighbourhood $\mathcal U_0$ of the origin in $\cX$ such that the zero-forcing problem has only the trivial solution in $\mathcal U_0$.
\end{corollary}

\begin{proof}
This is the case $\hf=0$ of the local uniqueness statement in Theorem~\ref{thm:main}. Equivalently, it follows from the holomorphic implicit function theorem applied above, since $\mathcal S_\C(0)=(0,0,0,0)$.
\end{proof}

\section{Numerical experiments}\label{sec:numerics}

The Kolmogorov $n$-width of a set $S\subset X$ is
\begin{equation}
  d_n(S;X):=\inf_{\substack{V_n\subset X\\ \dim V_n=n}}
  \sup_{u\in S}\inf_{v\in V_n}\|u-v\|_X.
\end{equation}
It measures the best worst-case error achievable by an $n$-dimensional linear space. The computations below do not evaluate this quantity directly: POD on a finite training set minimizes a mean-square projection error, and the reported test error is an empirical proxy in componentwise $L^2$ norms. Moreover, the analyticity theorem by itself does not imply exponential decay of $d_n$ for the infinite-dimensional force space; quantitative width estimates require additional finite-dimensional or sparsity assumptions. The purpose of the experiment is therefore to examine the observed linear compressibility of one finite-dimensional force family.

\subsection{Method and validation}

We solve the problem by the Picard iterative scheme until the maximal error between two iterations in all solution components is less than the given tolerance. The iteration is initialized at the trivial state $\hu^{(0)}=0$, $\eta^{(0)}=0$.
At iteration $k\ge 0$, given the previous iterate $(\hu^{(k)},\eta^{(k)})$, the next
iterate $(\hu^{(k+1)},\hp^{(k+1)},\eta^{(k+1)})$ is computed in three steps.

\vspace{0.5em}
\noindent\textbf{Step 1 -- the fluid problem:} Find
$(\hu^{(k+1)},\hp^{(k+1)}) \in H_0^1(\hat{\Omega})^2 \times L^2_0(\hat{\Omega})$
such that for all
\[ (\hat{v}, \hat{q}) \in H_0^1(\hat{\Omega})^2 \times L^2_0(\hat{\Omega}) \]
\begin{equation}\label{eq:weak-fluid-numeric}
\begin{aligned}
  &\int_{\hOm}2D_{\eta^{(k)}}(\hu^{(k+1)}):D_{\eta^{(k)}}(\hv)\,J_{\eta^{(k)}}\dd\hx
  +\int_{\hOm}\bigl((A_{\eta^{(k)}}\hu^{(k)})\cdot\nabla\bigr)\hu^{(k+1)}\cdot\hv\,J_{\eta^{(k)}}\dd\hx
  \\
  &\qquad
  -\int_{\hOm}\hp^{(k+1)}\,\diver_{\eta^{(k)}}\hv\,J_{\eta^{(k)}}\dd\hx
  =
  \int_{\hOm}\hf\cdot\hv\,J_{\eta^{(k)}}\dd\hx,
\end{aligned}
\end{equation}
\begin{equation}\label{eq:weak-div-numeric}
  \int_{\hOm}\diver_{\eta^{(k)}}\hu^{(k+1)}\,\hq\,J_{\eta^{(k)}}\dd\hx=0 .
\end{equation}

\vspace{0.5em}
\noindent\textbf{Step 2 -- beam problem:} To use $P_2$ elements, the beam problem is
reformulated as a first-order system. Find $(\eta_1,\psi_1)\in H_0^1(0,L)^2$ such that
for every $(\varphi,\rho)\in H_0^1(0,L)^2$
\begin{equation}\label{eq:weak-beam-numeric}
  \int_0^L\psi_1'\,\rho'\dd x
  =
  -\mathfrak T\bigl(\eta^{(k)},\hf,\hu^{(k+1)},\hp^{(k+1)}\bigr)(\varphi),
\end{equation}
\begin{equation}\label{eq:weak-beam-penalty}
  \alpha \int_0^L(\eta_1'-\psi_1)(\varphi'-\rho)\dd x = 0,
\end{equation}
with $\alpha=10^5$.

\vspace{0.5em}
\noindent\textbf{Step 3 -- volume constraint:} Determine the constant $c$ from
\begin{align*}
    c &= -\frac{\int_0^L \eta_1(x)\dd x}{\int_0^L \eta_0(x)\dd x}
       = -\frac{720}{L^5} \int_0^L \eta_1(x)\dd x,
\end{align*}
where $\eta_0(x)=\dfrac{x^2(L-x)^2}{24}$, and set
$\eta^{(k+1)}=\eta_1+c\,\eta_0$.
The iteration is stopped when the increment of every unknown is below the prescribed
tolerance.

The fluid subproblem is discretized with Taylor--Hood $P2$--$P1$ elements. In the fluid subproblem at iteration $k$, the convective term is linearized as
$(\hu^{(k)}\cdot\nabla)\hu^{(k+1)}$ and the geometric coefficients
$J_{\eta^{(k)}}$, $A_{\eta^{(k)}}$ are evaluated at the displacement $\eta^{(k)}$
from the previous iterate.  After solving the fluid problem, the beam problem is updated using the new traction. The fourth-order clamped beam equation is written as an equivalent mixed system of second-order equations and discretized with quadratic elements.

The implementation was first checked with a manufactured solution. On the reference domain, let $h=1+\eta$ and set
\begin{align*}
  \eta(x)
  &=\frac{1}{24}x^2(1-x)^2(2x-1),\\
  \hu_1(x,y)
  &=(\eta(x) + 1)^4x^2(1-x)^2y(1-y)^2(2-5y),\\
  \hu_2(x,y)
  &=-(\eta(x) + 1)^5y^2(1-y)^2\left[
      2x(1-x)(1-2x)(1-y)
      +\frac{3\eta'(x)}{h}x^2(1-x)^2
    \right],\\
  \hp(x,y)
  &=10\left(x-\frac12\right).
\end{align*}
A direct calculation gives $\diver_\eta\hu=0$. Moreover, $D_\eta(\hu)=0$ on the top boundary and $\eta^{(4)}=-\hp$, so the beam equation holds with $c=0$. The reference force is then obtained by substituting these exact fields into the transformed momentum equation. The resulting relative $L^2$ errors and observed refinement rates are reported in Table~\ref{tab:mms_convergence}.
\begin{table}[htbp]
  \centering
    \begin{tabular}{|l |cc |cc |c|}
        \hline
        & \multicolumn{2}{c}{$n=32$} & \multicolumn{2}{c}{$n=64$} & $n=128$ \\
        \hline
        Field & error & rate & error & rate & error \\
        \hline
        $u$    & $1.4420\mathrm{e}{-5}$ & 2.59 & $2.4018\mathrm{e}{-6}$ & 2.10 & $5.6044\mathrm{e}{-7}$ \\
        \hline
        $v$    & $1.4824\mathrm{e}{-5}$ & 2.44 & $2.7340\mathrm{e}{-6}$ & 2.06 & $6.5646\mathrm{e}{-7}$ \\
        \hline
        $p$    & $3.1750\mathrm{e}{-6}$ & 2.03 & $7.7549\mathrm{e}{-7}$ & 2.03 & $1.8953\mathrm{e}{-7}$ \\
        \hline
        $\eta$ & $1.1231\mathrm{e}{-2}$ & 2.01 & $2.7977\mathrm{e}{-3}$ & 2.03 & $6.8448\mathrm{e}{-4}$ \\
        \hline
    \end{tabular}
     \caption{Relative $L^2$ errors and observed convergence rates under uniform
  mesh refinement for the manufactured solution. The rate is calculated as $\log_2$ of the ratio of relative errors on the coarse and fine meshes. The observed rates are approximately second order.}
  \label{tab:mms_convergence}
\end{table}

We next consider a symmetric pair of counter-rotating localized forces in a closed elastic channel. Set
\[
  E_\pm(x,y):=
  \exp\!\left[-\frac{(x-(1/2\pm\mu))^2+(y-y_0)^2}{s^2}\right].
\]
The force family is
\begin{align}
  \hf_1(x,y)
  &=-A(y-y_0)\bigl(E_-(x,y)-E_+(x,y)\bigr),\\
  \hf_2(x,y)
  &=A\Bigl((x-(1/2-\mu))E_-(x,y)
    -(x-(1/2+\mu))E_+(x,y)\Bigr).
    \label{eq:forces}
\end{align}
The four parameters range over
\[
  y_0\in[0.225,2/3],\qquad
  \mu\in[0,0.5],\qquad
  A\in[5000,25000],\qquad
  s\in[0.25,0.35],
\]
where $x$ coordinate of the force centers is $\frac{1}{2}  \pm \mu $, $y_0$ corresponds to $y$ coordinates of the force centers, $A$ is amplitude and $s$ determines width of the force centers. 
Examples of force streamlines are shown in Figure \ref{fig_streamlines} with their corresponding solutions in Figure \ref{fig:solution}.
\begin{figure}
    \centering
    \includegraphics[width=0.45\linewidth]{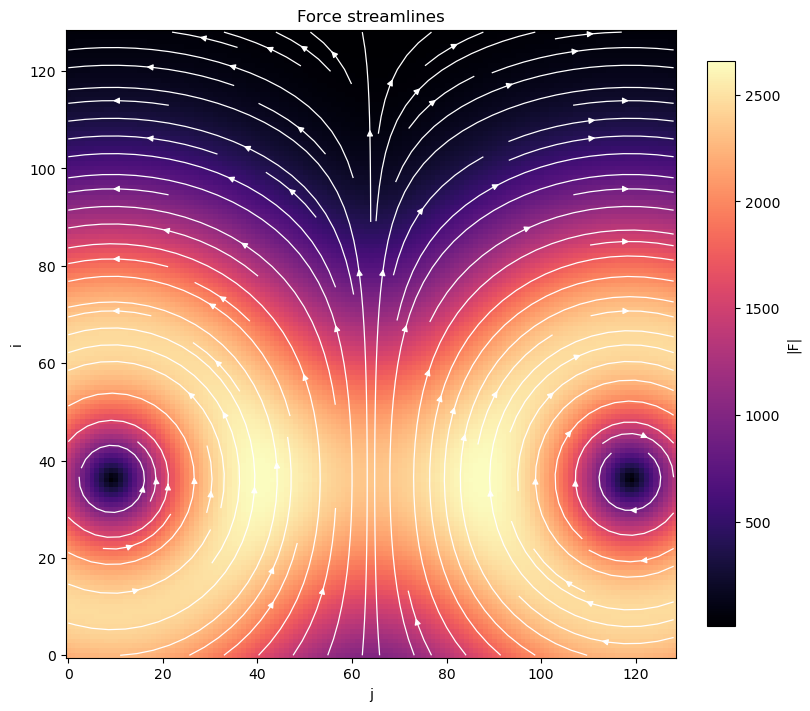}
    \includegraphics[width=0.45\linewidth]{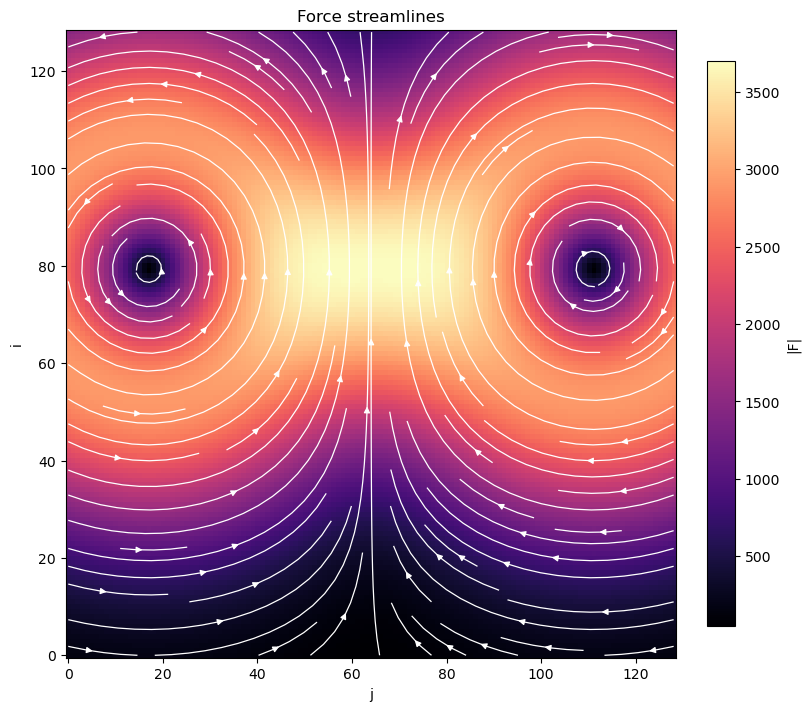}
    \caption{Examples of force streamlines with parameters $(y_0, \mu, A, s)$ equal to $(0.28, 0.43, 18600, 0.31)$ (left panel) and $(0.62, 0.37, 22950, 0.30)$ (right panel). }
    \label{fig_streamlines}
\end{figure}
\begin{figure}
    \centering
    \includegraphics[width=0.95\linewidth]{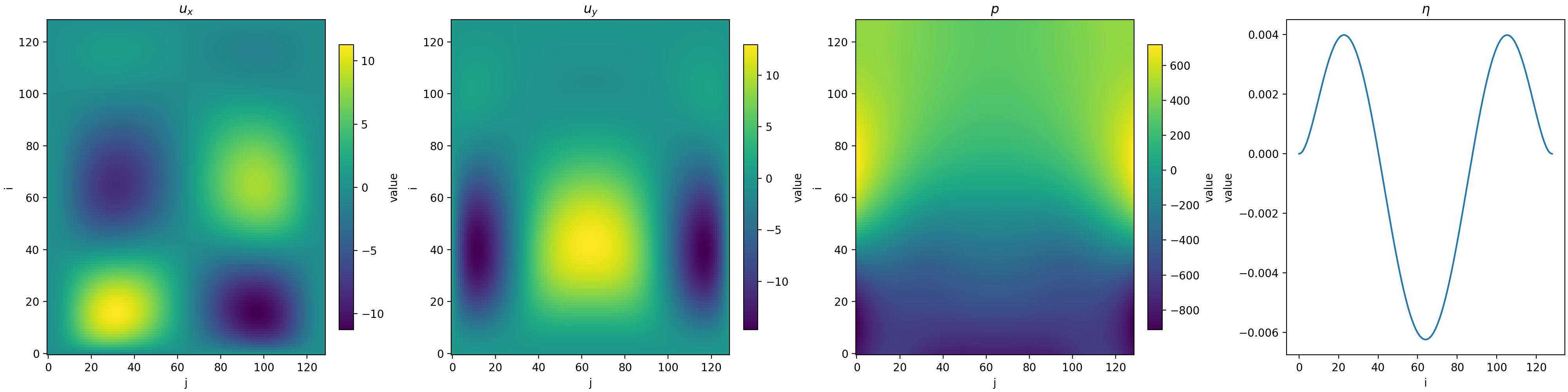}
    \includegraphics[width=0.95\linewidth]{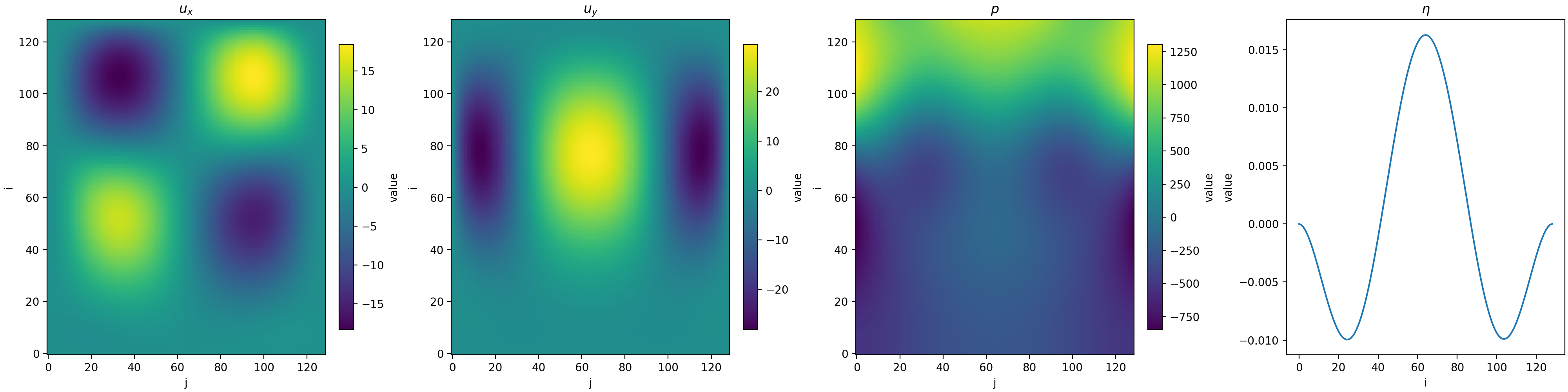}
    \caption{Solution for parameters $(y_0, \mu, A, s)$ equal to $(0.28, 0.43, 18600, 0.31)$ (first row) and $(0.62, 0.37, 22950, 0.30)$ (second row).}
    \label{fig:solution}
\end{figure}

The fluid mesh uses a $128\times128$ grid, and the structure mesh has $128$ points. 
The problem is solved with already described iterative scheme with tolerance $10^{-9}$. Across the computed data set, the maximum absolute boundary displacement is $6.8\%$ of the channel height and the average number of coupled Picard iterations is $15$.

We further assess mesh convergence for the force with $(y_0,\mu,A,s)=(0.66,0.3,15000,0.3)$, using the solution on a $512\times512$ fluid grid and a $512$-point structure mesh as reference. The results are shown in Table~\ref{tab:convergence}.

\begin{table}[htbp]
  \centering
\begin{tabular}{|l |cc |cc |c|}
    \hline
    & \multicolumn{2}{c}{$n=32$} & \multicolumn{2}{c}{$n=64$} & $n=128$ \\
    \hline
    Field & error & rate & error & rate & error \\
    \hline
    $u$    & $3.5300\mathrm{e}{-3}$ & 2.19 & $7.7561\mathrm{e}{-4}$ & 2.86 & $1.0679\mathrm{e}{-4}$ \\
    \hline
    $v$    & $3.5970\mathrm{e}{-3}$ & 2.24 & $7.6151\mathrm{e}{-4}$ & 2.87 & $1.04527\mathrm{e}{-4}$ \\    \hline
    $p$    & $5.5900\mathrm{e}{-3}$ & 2.20 & $1.2167\mathrm{e}{-3}$ & 2.62 & $1.9812\mathrm{e}{-4}$ \\
    \hline
    $\eta$ & $5.6996\mathrm{e}{-2}$ & 2.21 & $1.2331\mathrm{e}{-2}$ & 3.13 & $1.4123\mathrm{e}{-3}$ \\
    \hline
\end{tabular}
  \caption{Relative $L^2$ errors and observed convergence rates under uniform
  mesh refinement, using the $n=512$ solution as reference. Rates are computed
  between successive meshes. The $n=256$ level is omitted: being only one
  refinement below the reference, its error is contaminated by the reference
  solution's own discretization and linear-solver error and does not reflect
  the true discretization error. The reported rates are at least second order over these refinements.}
  \label{tab:convergence}
\end{table}
The manufactured-solution test (Table \ref{tab:mms_convergence}) exhibits approximately second-order convergence in the relative $L^2$ errors, validating the implementation. The self-referential refinement study (Table \ref{tab:convergence}) is consistent with this but, using the finest mesh as reference, is not a clean asymptotic-rate measurement.

\subsection{Reduced-order approximability}
We now examine how well the solution manifold generated by the four-parameter force family is approximated by low-dimensional linear subspaces. Rather than the worst-case Kolmogorov width, we measure the empirical mean-square reconstruction error of a POD basis on held-out data, which is the quantity relevant to data-driven reduced-order models. The aim is to gauge the linear compressibility of the manifold for this particular family, not to certify a width rate.

We use the method of snapshots with proper orthogonal decomposition (POD) \cite{c01a2bc6-f9f6-36a9-835f-6f8e5dd97c30}. For each component, the centered training snapshots are vectorized and assembled as the columns of a matrix $A$. If $A=U\Sigma V^T$ is its singular value decomposition, the first $k$ columns of $U$ span the $k$-dimensional space minimizing the mean squared projection error over this finite training set \cite{quarteroni2015reduced}. This is not the worst-case optimality appearing in the definition of the Kolmogorov width. 

For a finite-dimensional force family $\hf = (\hf_1, \hf_2)$ as in \eqref{eq:forces}, depending analytically on 
$(y_0,\mu,A,s)$ the composition with the solution map is holomorphic on the parameter domain wherever Theorem \ref{thm:main} applies. On any compact parameter set admitting an explicit polyellipse of holomorphy, the standard finite-dimensional holomorphy argument \cite{cohenDeVore2015} then yields exponential $n$-width decay. We verify neither the required domain-of-holomorphy bounds nor the smallness of Theorem \ref{thm:main}, so the decay in Figure \ref{podDecay} is reported as empirical.

POD modes are computed from a training set of $200$ solutions, and the mean relative $L^2$ reconstruction error is evaluated on a test set of $100$ solutions. Modes are computed separately for each solution component after subtracting its training-set mean. Figure~\ref{podDecay} shows the test error as a function of the number of modes.

\begin{figure}
    \centering
    \includegraphics[width=0.75\linewidth]{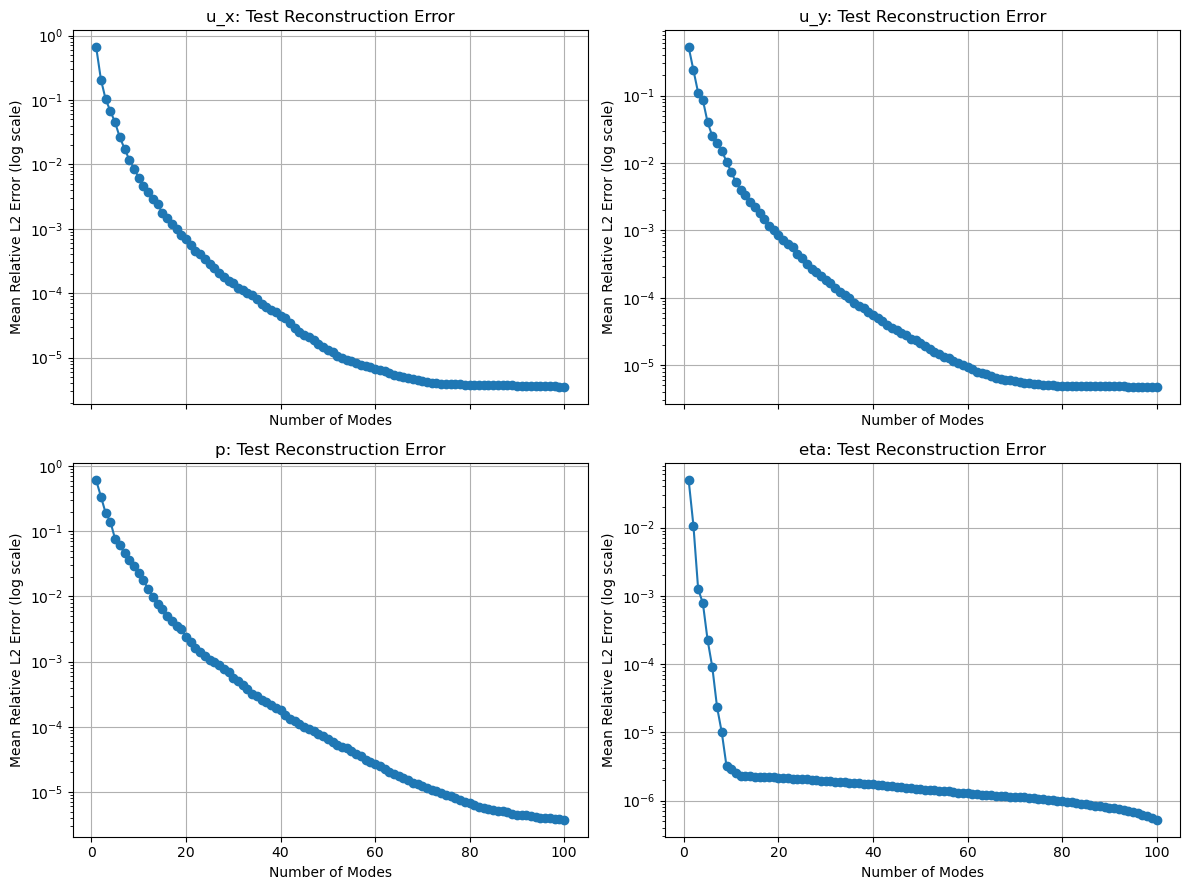}
    \caption{Mean relative $L^2$ POD reconstruction error on the test set in log scale for the two velocity components (top row), pressure, and $\eta$ (bottom row). The errors decrease rapidly and then reach floors between approximately $10^{-5}$ and $10^{-6}$. These floors may reflect a combination of spatial discretization, nonlinear-solver tolerance, finite-sample effects, and numerical linear algebra. The observed decay is empirical; the analyticity theorem alone does not furnish an exponential Kolmogorov-width rate for this force family.}
    \label{podDecay}
\end{figure}

\section{Conclusion}

We proved that the weak solution of the stationary Navier--Stokes--beam problem depends real analytically on the reference-domain volume force in a neighbourhood of the trivial state. As a consequence, the analysis also provides local existence and uniqueness for sufficiently small data.

The numerical scheme was implemented and verified by means of a manufactured-solution test. For the considered parametric family of localized forces, the POD approximation errors exhibit an approximately exponential decay over the resolved range before reaching a numerical floor. This provides empirical evidence of a rapidly decaying Kolmogorov approximation rate for the sampled solution manifold. Such behaviour is consistent with the expected reduced-order approximability of parametrically analytic solution maps, and the analyticity result established in this work provides a theoretical foundation in this direction.

\section*{Acknowledgements}
This research was supported by the Croatian Science Foundation under project IP-2022-10-2962. Boris Muha was supported by the European Union -- NextGenerationEU through the National Recovery and Resilience Plan 2021--2026 and by the institutional grant of the University of Zagreb Faculty of Science IK IA 1.1.3 Impact4Math.

During the preparation of this work, the authors used OpenAI's ChatGPT for language editing, assistance with drafting and revising parts of the manuscript, and literature searches. All AI-assisted content, including references identified with its assistance, was independently reviewed and verified by the authors. The authors take full responsibility for the content of the manuscript.

\bibliographystyle{alpha}
\bibliography{references_regularized_reviewed}

\end{document}